 \documentclass[final,1p,times]{elsarticle}

\usepackage{amssymb}
 \usepackage{amsthm}
\usepackage{amsmath,amssymb,amsopn,amsfonts,mathrsfs,amsbsy,amscd}

\makeatletter

\@addtoreset{equation}{section}
\makeatother

\newcommand{\prs}{\langle\;,\;\rangle}
\newcommand{\too}{\longrightarrow}

\newcommand{\esp}{\quad\mbox{and}\quad}

\newcommand{\G}{{\mathfrak{g}}}

\newcommand{\h}{{\mathfrak{h}}}
\newcommand{\ad}{{\mathrm{ad}}}

\newcommand{\al}{\alpha}

\font\bb=msbm10

\def\R{\hbox{\bb R}}

\def\N{\hbox{\bb N}}

\newtheorem{theo}{Theorem}[section]
\newtheorem{pr}{Proposition}[section]
\newtheorem{Le}{Lemma}[section]
\newtheorem{co}{Corollary}[section]

\newtheorem{exem}{Example}
\newtheorem{rem}{Remark}

\begin{document}

\begin{frontmatter}




\title{On Flat Pseudo-Euclidean Nilpotent Lie Algebras}

 \author[Mohamed Boucetta]{Mohamed Boucetta}
 \address[Mohamed Boucetta]{Universit\'e Cadi-Ayyad,
  Facult\'e des sciences et techniques,
  B.P. 549, Marrakech, Maroc.\\m.boucetta@uca.ac.ma
 }
 \author[Hicham Lebzioui]{Hicham Lebzioui\corref{mycorrespondingauthor}}
 \cortext[mycorrespondingauthor]{Corresponding author}
\address[Hicham Lebzioui]{Universit\'e Moulay Sma\"{\i}l,
 \'{E}cole Sup\'{e}rieur de Technologie Kh\'{e}nifra, B.P : 170, Kh\'{e}nifra, Maroc.\\h.lebzioui@estk.umi.ac.ma
}


\begin{abstract}A flat pseudo-Euclidean Lie algebra is a real Lie algebra with a non degenerate symmetric bilinear form and a  left symmetric product whose the commutator is the Lie bracket and  such that the left multiplications are skew-symmetric. We  show that the center of a  flat pseudo-Euclidean nilpotent Lie algebra  of signature $(2,n-2)$  must be degenerate and  all flat pseudo-Euclidean nilpotent Lie algebras of signature $(2,n-2)$ can be obtained by using the double extension process from flat Lorentzian nilpotent Lie algebras. 
We show also that the center of a flat pseudo-Euclidean 2-step nilpotent Lie algebra is degenerate and all these Lie algebras are obtained by using a sequence of double extension from an abelian Lie algebra. In particular, we determine all flat pseudo-Euclidean 2-step nilpotent Lie algebras of signature $(2,n-2)$. The paper contains also some examples in low dimension.  
\end{abstract}


\end{frontmatter}






\section{Introduction}
A {\it flat pseudo-Euclidean Lie algebra} is a real Lie algebra with a non degenerate symmetric bilinear form and a  left symmetric product whose the commutator is the Lie bracket and   such that the left multiplications are skew-symmetric. In geometrical terms, a flat pseudo-Euclidean Lie algebra is the Lie algebra of a Lie group with a left-invariant pseudo-Riemannian metric with vanishing curvature. Let $(\G,\prs)$ be a flat pseudo-Euclidean Lie algebra of dimension $n$. If the metric $\prs$ is definite positive (resp. of signature $(1,n-1)$), then $(\G,\prs)$ is called  Euclidean  (resp. Lorentzian). Flat pseudo-Euclidean Lie algebras have been studied mostly in the Euclidean and the Lorentzian cases. 
Let us enumerate some important results on flat pseudo-Euclidean Lie algebras: \begin{enumerate}

\item In \cite{Milnor}, Milnor showed that $(\G,\prs)$ is a flat Euclidean Lie algebra if and only if $\G$ splits orthogonally as $\G=\mathfrak{b}\oplus\mathfrak{u}$, where $\mathfrak{u}$ is an abelian ideal, $\mathfrak{b}$ is an abelian subalgebra, and $\ad_b$ is skew-symmetric for any $b\in\mathfrak{b}$. According to this theorem, a nilpotent (non-abelian) Lie algebra can not admit a flat Euclidean metric.

\item In \cite{Aubert}, Aubert and Medina 
showed that all flat Lorentzian nilpotent Lie algebras are obtained by the double extension process from
Euclidean abelian Lie algebras. 
\item Gu\'{e}diri showed in
\cite{guediri} that a flat  Lorentzian 2-step nilpotent Lie algebra is a
trivial extension  of the 3-dimensional Heisenberg Lie algebra
$\mathcal{H}_3$.

 \item In \cite{ABL, AB}, the authors  showed that  flat  Lorentzian  Lie
 algebras with degenerate center or flat nonunimodular Lorentzian  Lie algebras can be obtained by the double extension process from  flat Euclidean Lie
 algebras.\end{enumerate}

The study of flat pseudo-Euclidean Lie algebras of signature other than $(0,n)$ and $(1,n-1)$ is  an open  problem. In this paper, we tackle a part of this problem, namely, we study 
flat pseudo-Euclidean nilpotent Lie algebras of signature $(2,n-2)$ and flat  pseudo-Euclidean 2-step nilpotent Lie algebras of any signature. 
 There are our main results:
\begin{enumerate}
	\item In Theorem \ref{degenerate}, we show that the center of a flat pseudo-Euclidean nilpotent Lie algebra of signature $(2,n-2)$ must be degenerate. From this theorem and Theorem \ref{twosided} we deduce that  all flat pseudo-Euclidean nilpotent Lie algebra of signature $(2,n-2)$ are obtained by the double extension process.
	\item We give some general properties of  flat  pseudo-Euclidean 2-step nilpotent Lie algebras and we show that their center is degenerate.
	we show also that we can construct all this Lie algebras by applying a sequence of double extension starting from a pseudo-Euclidean abelian Lie algebra.
	
	\item We give all 2-step nilpotent Lie algebras which can admit flat pseudo-Euclidean metrics of signature $(2,n-2)$ (Theorem \ref{dim1} and Theorem \ref{dim2}). We will see that a class of 2-step nilpotent Lie algebras which can admit a flat pseudo-Euclidean metrics of signature $(2,n-2)$ is very rich, contrary to the Euclidean and the Lorentzian cases. As example, we show that any 6-dimensional 2-step nilpotent Lie algebra which is not an extension trivial of a 5-dimensional Heisenberg Lie algebra, admits such metric.
	
\end{enumerate}

The paper is organized as follows. In section 2, we give some
generalities on flat pseudo-Euclidean Lie algebras. In section 3 and section 4, we
study flat pseudo-Euclidean metrics of signature $(2,n-2)$ on nilpotent Lie algebras . In
section 5, we study flat pseudo-Euclidean 2-step nilpotent Lie algebra of any signature. In section 6, we give all flat pseudo-Euclidean 2-step nilpotent Lie algebras of signature $(2,n-2)$. We end the paper by giving some examples.

\section{Preliminaries}\label{section2}

In this section, we give some general results on nilpotent Lie algebras and on flat pseudo-Euclidean nilpotent Lie algebras which will be crucial in the proofs of our main results.

Let us start with two useful lemmas. Recall that a pseudo-Euclidean vector space is a real finite dimensional vector space endowed with a non degenerate bilinear symmetric form.
\begin{Le}\label{le1}
	Let $(V,\prs)$ be a pseudo-Euclidean vector space and $A$ a skew-symmetric endomorphism satisfying $A^2=0$ and $\dim \mathrm{Im}A\leq1$. Then $A=0$.
\end{Le}
\begin{proof} Suppose that $A\not=0$. Then $\mathrm{Im}A$ is a totally isotropic vector space of dimension
1. This implies that $\ker A$ is an hyperplan which contains $\mathrm{Im}A$. Let $e$ be a generator of $\mathrm{Im}A$ and choose an isotropic vector $\bar{e}\notin\ker A$ such that $\langle e,\bar{e}\rangle=1$. We have $V=\ker A\oplus\R \bar{e}$ and $A(\bar{e})=\al e$. Then
$\al=\langle A(\bar{e}),\bar{e}\rangle=0$ which gives a contradiction and completes the proof. \end{proof}

\begin{Le}\label{codimension}
	Let $\G$ be a nilpotent Lie algebra,  $\mathfrak{a}$ and $\mathfrak{h}$, respectively, a Lie subalgebra of codimension one and an ideal of codimension two. Then $[\G,\G]$ is contained in $\mathfrak{a}$ and in $\mathfrak{h}$.
\end{Le}
\begin{proof}
	We have $\G/\mathfrak{h}$ is a 2-dimensional nilpotent Lie algebra and hence must be abelian. This implies that $[\G,\G]\subset\mathfrak{h}$. On the other hand, write $\G=\mathfrak{a}\oplus\R y$. For any $x\in\mathfrak{a}$, we have
	$$[x,y]=a(x)y+u_1,\;\mbox{where }u_1\in\mathfrak{a}.$$
	Since $\mathfrak{a}$ is a Lie subalgebra then, for any $n\in\N^*$, $\ad^n_x(y)=a(x)^ny+u_n$ with $u_n\in\mathfrak{a}$. Since $\ad_x$ is nilpotent then $a(x)=0$ and the result follows.
\end{proof}
We pursue with some general properties of flat pseudo-Euclidean Lie algebras. A pseudo-Euclidean Lie algebra $(\G,\prs)$ is a finite dimensional real Lie algebra $\G$ endowed with a non degenerate symmetric bilinear form $\prs$. We define a product $(u,v)\mapsto u.v$ on $\G$ called Levi-Civita product by  Koszul's formula
\begin{equation}\label{lc}
2\langle
u.v,w\rangle=\langle[u,v],w\rangle+\langle[w,u],v\rangle+\langle[w,v],u\rangle,
\end{equation}
for any $u,v,w\in\G$. We denote by $\mathrm{L}_u:\G\too\G$ and $\mathrm{R}_u:\G\too\G$,
respectively, the left multiplication and the right multiplication by $u$ given by
$\mathrm{L}_uv=u.v$ and $\mathrm{R}_uv=v.u.$ For any $u\in\G$, $\mathrm{L}_u$ is
skew-symmetric with
respect to $\prs$ and
$\ad_u=\mathrm{L}_u-\mathrm{R}_u,$ where
$\ad_u:\G\too\G$ is given by $\ad_uv=[u,v]$. We call $(\G,\prs)$  {\it flat pseudo-Euclidean Lie algebra} if the Levi-Civita product is left symmetric, i.e., for any $u,v,w\in\G$,
\begin{equation}\label{associative}
\mathrm{ass}(u,v,w)=\mathrm{ass}(v,u,w),
\end{equation}
where
$\mathrm{ass}(u,v,w)=(u.v).w-u.(v.w).$

\begin{rem} Let $G$ be a Lie group, and $\mu$ a left-invariant pseudo-Riemannian metric on $G$. Let $\G=\mbox{Lie}(G)$ and $\prs=\mu_e$. Then the curvature of $(G,\mu)$ vanishes if and only if $(\G,\prs)$ is a flat pseudo-Euclidean Lie algebra.\end{rem}

Let $(\G,\prs)$ be a flat pseudo-Euclidean Lie algebra. The condition \eqref{associative} is also equivalent  to one of the following relations:
\begin{eqnarray}
&&\mathrm{L}_{[u,v]}=[\mathrm{L}_u,\mathrm{L}_v],\label{eqs1}\\
&&\mathrm{R}_{u.v}-\mathrm{R}_v\circ \mathrm{R}_u=[\mathrm{L}_u,\mathrm{R}_v],\label{eqs3}
\end{eqnarray}for any $u,v\in\G$. We denote by $Z(\G)=\{u\in\G,\ad_u=0  \}$ the center of $\G$. For any $u,v\in Z(\G)$ and $a,b\in\G$, one can deduce easily from \eqref{lc}-\eqref{eqs3} that
\begin{equation}\label{center}
u.v=0,\; \mathrm{L}_u=\mathrm{R}_u,\; \mathrm{L}_u\circ \mathrm{L}_v=0\esp u.(a.b)=a.(u.b).
\end{equation}

\begin{pr}\label{pr2}
	Let $(\G,\prs)$ be a flat pseudo-Euclidean nilpotent non abelian Lie algebra. If $Z(\G)=\{u\in\G, \mathrm{L}_u=\mathrm{R}_u=0 \}$ then $Z(\G)$ is degenerate.
\end{pr}
\begin{proof}  One can see easily that the
	orthogonal of
	the derived ideal of $\G$ is given by
	\begin{equation}
	\label{eq7}[\G,\G]^\perp=\{u\in\G,\mathrm{R}_u=\mathrm{R}_u^*\}.
	\end{equation}
	 Then $Z(\G)\subset[\G,\G]^\perp$ and hence $[\G,\G]\subset Z(\G)^\perp$. Since $\G$ is nilpotent non abelian then
$\{0\}\not= [\G,\G]\cap Z(\G)\subset Z(\G)^\perp\cap Z(\G)$. This shows that $Z(\G)$ is degenerate.
\end{proof}

\begin{pr} \label{pr3} Let $(\G,\prs)$ be a flat pseudo-Euclidean nilpotent Lie algebra. Then:
	\begin{enumerate}
		\item If $(\G,\prs)$ is Euclidean then $\G$ is abelian.
		\item If $(\G,\prs)$ is non abelian Lorentzian then $Z(\G)$ is degenerate.
	\end{enumerate}
\end{pr}
\begin{proof} \begin{enumerate}
	\item According to \eqref{center}, for any $u\in Z(\G)$, $\mathrm{L}_u$ is a nilpotent skew-symmetric endomorphism and hence must vanishes. This gives the result, by virtue of Proposition \ref{pr2}.
	\item This is a consequence of \eqref{center}, Lemma \ref{le1} and Proposition \ref{pr2}.  \qedhere 
\end{enumerate}\end{proof}
Put $N(\G)=\bigcap_{u\in Z(\G)}\ker\mathrm{L}_u$, $\G_0:=N(\G)\cap Z(\G)^\perp$ and $\h_0:=N(\G)^\perp$. These vector spaces and the following lemma which states their main properties will play a central role in this paper, namely, in the proof of Theorem \ref{degenerate}.
\begin{Le}\label{crucial} Let $(\G,\prs)$ be a flat pseudo-Euclidean nilpotent Lie algebra of signature $(2,n-2)$, $n\geq4$. 
	 Then:\begin{enumerate}
	 	\item 
	  $N(\G)$, $\G_0$ and $\h_0$ are left ideals for the Levi-Civita product,   $\h_0\subset\G_0$, and $\h_0$ is  totally isotropic with $\dim\h_0\leq2$. \item If $Z(\G)$ is non degenerate then  the restriction of $\prs$ to $Z(\G)$ is positive definite, $\dim \h_0=2$ and $\dim(Z(\G)\cap[\G,\G])=1$. Moreover,  if $z_0$ is a  generator of $Z(\G)\cap[\G,\G]$ with $\langle z_0,z_0\rangle=1$ then  for any $u,v\in\G$,
	\begin{equation}\label{bra3}
	[u,v]=[u,v]_1-2\langle \mathrm{L}_{z_0}u,v\rangle z_0,
	\end{equation}
	where $[u,v]_1\in Z(\G)^\perp$.
\end{enumerate}

\end{Le}

\begin{proof}\begin{enumerate}
	\item Note first that, for any $u\in\G$, $(\ker \mathrm{L}_u)^\perp=\mathrm{Im}\mathrm{L}_u$ and hence $\h_0=\sum_{u\in Z(\G)}\mathrm{Im}\mathrm{L}_u$.   From \eqref{center}, we have clearly that $Z(\G)\subset N(\G)$ and,  for any $u,v\in Z(\G)$, $\mathrm{Im}\mathrm{L}_u\subset \ker \mathrm{L}_v$. Thus  $\h_0\subset\G_0$. This implies that $\h_0$ is totally isotropic and since the signature is $(2,n-2)$ one must have $\dim\h_0\leq 2$. One can deduce easily from the third relation in \eqref{center} that $N(\G)$ is a left ideal. This implies, since the left multiplication are skew-symmetric that $\h_0$ and $\G_0$ are also  left ideals.

	\item Suppose now that $Z(\G)$ is non degenerate.  If $\dim\h_0\leq1$ then, according to Lemma \ref{le1},  $\mathrm{L}_u=0$ for any $u\in Z(\G)$  and hence,  by virtue of  Proposition \ref{pr2}, $Z(\G)$ is degenerate. So we must have   $\dim\h_0=2$ and  the restriction of $\prs$ to $Z(\G)^\bot$ is of signature $(2,\dim Z(\G)^\bot-2)$ which implies that the restriction of $\prs$ to $Z(\G)$ is definite positive.
	On the other hand, according to what above we can choose   two  vectors $(\bar{e}_1,\bar{e}_2)$ of $Z(\G)^\perp$ such that
	 $Z(\G)^\perp=\G_0\oplus\mathrm{Span}\{ \bar{e}_1,\bar{e}_2 \}.$ So,
	\[ [\G,\G]=[Z(\G)^\perp,Z(\G)^\perp]=\R[\bar{e}_1,\bar{e}_2]+[\bar{e}_1,\G_0]+[\bar{e}_2,\G_0]+[\G_0,\G_0]. \]We have that $\G_0$ is a left ideal for the Levi-Civita product and for any $a\in\G_0$, $b\in\G$ and $u\in Z(\G)$,
	\[ \langle a.b,u\rangle=-\langle b,a.u\rangle=-\langle b,u.a\rangle=0 \]and hence
	$\G_0.\G\subset Z(\G)^\perp$. This implies that $[\bar{e}_1,\G_0]+[\bar{e}_2,\G_0]+[\G_0,\G_0]\subset Z(\G)^\perp$. Moreover, $[\bar{e}_1,\bar{e}_2]=z+v_0,$ where $z\in Z(\G)$, $z\not=0$ since $Z(\G)\cap[\G,\G]\not=0$ and $v_0\in Z(\G)^\perp$. So $[\G,\G]=\R z\oplus F$ where $F$ is a vector subspace of $Z(\G)^\perp$. From this relation, we can deduce that $Z(\G)\cap[\G,\G]=\R z$ and \eqref{bra3} follows immediately.\qedhere
\end{enumerate}\end{proof}

\section{The center of a flat pseudo-Euclidean nilpotent Lie algebra of signature $(2,n-2)$ is degenerate} \label{section3}

The purpose of this section is to prove the following theorem.
\begin{theo}\label{degenerate} Let $(\G,\prs)$ be a flat pseudo-Euclidean nilpotent non abelian Lie algebra  of signature $(2,n-2)$ with $n\geq4$. Then $Z(\G)$ is degenerate.
\end{theo}
\begin{proof} We proceed by contradiction and we suppose that
	 $Z(\G)$ is non degenerate, i.e.,
	$ \G=Z(\G)\oplus Z(\G)^\perp$. As in Lemma \ref{crucial}, we consider
	 $\G_0=\{v\in Z(\G)^\perp/ \mathrm{L}_uv=0,\;\forall u\in Z(\G) \}$ and $\h_0$ its orthogonal in $Z(\G)^\bot$. We have both $\h_0$ and $\G_0$ are left ideals for the Levi-Civita product, $\h_0\subset\G_0$ and $\h_0$ is totally isotropic of dimension 2. Moreover, if
	 $z_0$ is a unit generator of $Z(\G)\cap[\G,\G]$ then, for any $u,v\in\G$,
	\begin{equation}\label{bra30}
	[u,v]=[u,v]_1-2\langle \mathrm{L}_{z_0}u,v\rangle z_0,
	\end{equation}
	where $[u,v]_1\in Z(\G)^\perp$. This relation shows that $\mathrm{L}_{z_0}\not=0$ and since $\mathrm{L}_{z_0}^2=0$ and $\mathrm{Im}\mathrm{L}_{z_0}\subset\h_0$, by virtue of Lemma \ref{le1}, $\mathrm{Im}\mathrm{L}_{z_0}=\h_0$ and $\ker \mathrm{L}_{z_0}=Z(\G)\oplus\G_0$. Moreover, from \eqref{bra30}, one can check easily that 
	 $[\;,\;]_1$ satisfies Jacobi identity and
	$(Z(\G)^\perp,[\;,\;]_1)$ becomes a nilpotent Lie algebra. We denote by $\circ$ the Levi-Civita product of $(Z(\G)^\perp,[\;,\;]_1,\prs)$ and we have obviously, for any $u,v\in Z(\G)^\perp$,
	\begin{equation} u.v=u\circ v-\langle \mathrm{L}_{z_0}u,v\rangle z_0. \label{eq50}\end{equation}
	Let $C(\G)$ denote the center of $(Z(\G)^\perp,[\;,\;]_1)$. We have $C(\G)\not=0$ and $C(\G)\cap\G_0=\{0\}$. Indeed, if $u\in C(\G)\cap\G_0$, then for any $v\in Z(\G)^\perp$,
	\[ [u,v]=[u,v]_1-2\langle \mathrm{L}_{z_0}u,v\rangle z_0=0, \]  hence $u\in Z(\G)$ and then  $u=0$. This implies that $1\leq\dim C(\G)\leq 2$ and for any $u\in C(\G)\setminus\{0\}$, $z_0.u\not=0$.
	
	Let $z$ be a non-null vector in $C(\G)$ then $z_0.z$ is a non-null vector in
	$\h_0$. From \eqref{eqs1} we get $\mathrm{L}_{z}\circ \mathrm{L}_{z_0}=\mathrm{L}_{z_0}\circ \mathrm{L}_{z}$ and by using \eqref{eqs3}  we have
	\[  \mathrm{R}_{z.z_0}=\mathrm{R}_{z_0}\circ \mathrm{R}_{z}=
	\mathrm{L}_{z_0}\circ \mathrm{R}_{z}.\]
	For any $u\in Z(\G)^\bot$, we have from \eqref{eq50} and the fact that $z\in C(\G)$,
	\[ \mathrm{L}_zu=z\circ u-\langle {z_0}.z,u\rangle z_0\esp \mathrm{R}_zu=u\circ z+\langle {z_0}.z,u\rangle z_0=z\circ u+\langle {z_0}.z,u\rangle z_0 . \]Thus
	$\mathrm{L}_zu=\mathrm{R}_zu-2\langle {z_0}.z,u\rangle z_0$. This relation is also true for $u\in Z(\G)$ since $z_0.u=0$ and hence $\mathrm{L}_z=\mathrm{R}_z+A_z$,
	where $\mathrm{A}_z=-2\langle z.z_0,.\rangle z_0$. Since $\mathrm{L}_{z_0}\circ A_z=0$, we deduce that\begin{equation}\label{new1}
	 \mathrm{R}_{z.z_0}=\mathrm{L}_{z_0}\circ \mathrm{R}_{z}=\mathrm{L}_{z_0}\circ(\mathrm{L}_z-A_z)=  \mathrm{L}_{z_0}\circ \mathrm{L}_{z}=\mathrm{L}_{z}\circ \mathrm{L}_{z_0}.\end{equation} This relation implies that $\mathrm{R}_{z.z_0}$ is symmetric and $\G_0\oplus Z(\G)\subset\ker\mathrm{R}_{z.z_0}$. From \eqref{eq50}, we have $z.z=0$, and hence $\G_0\oplus\R z\oplus Z(\G)\subset\ker \mathrm{R}_{z.z_0}$. From the symmetry of $\mathrm{R}_{z.z_0}$ we deduce that  $\mathrm{Im} \mathrm{R}_{z.z_0}=(\ker \mathrm{R}_{z.z_0})^\bot$ and finally $\mathrm{Im} \mathrm{R}_{z.z_0}\subset (\G_0\oplus\R z\oplus Z(\G))^\bot=\R z.z_0$. So we can write, for any $u\in\G$, 
	 \begin{equation}\label{new}
	 \mathrm{R}_{z.z_0}(u)=a_1(u)z.z_0=\al\langle z.z_0,u\rangle z.z_0,
	 \end{equation}where $a_1\in\G^*$ and $\al\in\R$. We will show now that $\mathrm{R}_{z.z_0}=0$.
	
	Put $e_1=z_0.z$. Since the orthogonal of $z$ in $Z(\G)^\perp$ is different from the orthogonal of $e_1$ in $Z(\G)^\perp$, we can choose  $\bar{z}\in Z(\G)^\perp$ such that $\langle z,\bar{z}\rangle=0$ and $\langle e_1,\bar{z}\rangle=1$. We put
	 $e_2=-z_0.\bar{z}$. We have $\langle e_2,z\rangle=1$, $Z(\G)^\perp=\G_0\oplus\mathrm{span}\{z,\bar{z}\}$ and
	$(e_1,e_2)$ is a basis of $\h_0$.  Now $\h_0$ is a 2-dimensional subalgebra of a nilpotent Lie algebra then it must be abelian and since $\h_0\subset\ker\mathrm{R}_{e_1}$ we deduce that $e_1.e_1=e_1.e_2=e_2.e_1=0$. Moreover,  $\h_0$ is a left ideal and we can write, for any $u\in\G$,
	\[ u.e_1=a_1(u)e_1\esp u.e_2=a_2(u)e_1+b_2(u)e_2. \]From the relation $u.(z_0.z)=z_0.(u.z)$ shown in  \eqref{center}, we deduce that $a_1(u)z_0.z=z_0.(u.z)$,  $a_1(u)z-u.z\in \ker\mathrm{L}_{z_0}=\h_0^\perp$ and hence
	\[ 0=a_1(u)\langle z,e_2\rangle-\langle u.z,e_2\rangle=a_1(u)\langle z,e_2\rangle+\langle z,u.e_2\rangle=a_1(u)+b_2(u). \]
	 Thus $b_2=-a_1$.
		Using the fact that the curvature vanishes, we get
		\begin{eqnarray*}
			\;[u,v].{e_2}&=&u.(v.{e_2})-v.(u.{e_2})\\&=&u.(a_2(v){e_1}-a_1(v){e_2})-v.(a_2(u){e_1}-a_1(u){e_2})\\
			&=&2(a_2(v)a_1(u)-a_1(v)a_2(u))e_1.
		\end{eqnarray*}Thus
		\begin{equation*}\label{a2}
		a_2([u,v])=2(a_2(v)a_1(u)-a_1(v)a_2(u)).
		\end{equation*}
	By taking $u=z$ and $v=\bar{z}$ in this relation and since $a_2(z_0)=0$,  $a_1(z)=0$ and, by virtue  of \eqref{bra30}, $[z,\bar{z}]=-2z_0$, we get
	$a_2(z)a_1(\bar{z})=0$. Now
	\[ a_1(\bar{z})e_1=\mathrm{R}_{e_1}(\bar{z})\stackrel{\eqref{new1}}{=}\mathrm{L}_{z}\circ \mathrm{L}_{z_0}(\bar{z})=-z.e_2=-a_2(z)e_1. \]
	This relation and $a_2(z)a_1(\bar{z})=0$ imply that $\mathrm{R}_{e_1}(\bar{z})=0$. But $\G_0\oplus\R z\oplus Z(\G)\subset\ker \mathrm{R}_{e_1}$ so finally 
	$\mathrm{R}_{e_1}=0$.
	 To complete, we will show that $e_1\in Z(\G)$, i.e, $\mathrm{L}_{e_1}=\mathrm{ad}_{e_1}=0$ and we will get a contradiction. \\
	Note first that $\mathrm{L}_{e_1}$ is nilpotent,  $\mathrm{L}_{e_1}(\h_0)=0$ and $\mathrm{L}_{e_1}(\G_0)\subset\G_0$. So $\mathrm{L}_{e_1}$ induces on the Euclidean vector space $\G_0/\h_0$ a skew-symmetric nilpotent endomorphism which must then vanish. So $\mathrm{L}_{e_1}(\G_0)\subset\h_0$. On the other hand, by virtue of \eqref{bra30}, 
	$e_1.z=[e_1,z]=0$. So for any $x\in\G_0$, $e_1.x=[e_1,x]=a(x)e_1+b(x)e_2$. This implies that $b(x)=\langle e_1.x,z\rangle=-\langle x,e_1.z\rangle=0$. But
	 $\ad_x$ is nilpotent so $a(x)=0$ and we deduce that $\mathrm{L}_{e_1}(\G_0)=0$. So far, we have shown that $\G_0\oplus\R z\oplus Z(\G)\subset\ker\mathrm{L}_{e_1}$ and hence its image has a dimension less or equal to 1. Moreover, $\mathrm{Im}\mathrm{L}_{e_1}\subset \h_0$ and hence $\mathrm{L}_{e_1}^2=0$ and we can conclude by using Lemma \ref{le1}.
\end{proof}

\section{Flat pseudo-Euclidean nilpotent Lie algebras of signature $(2,n-2)$ are obtained by the double extension process}\label{section4}

In this section, based on Theorem \ref{degenerate}, we will show that any flat pseudo-Euclidean nilpotent Lie algebra of signature $(2,n-2)$ can be obtained by the double extension process from a Lorentzian or an Euclidean flat nilpotent Lie algebra. To do so we need first to recall the double extension process introduced by Aubert and Medina \cite{Aubert}. Note that Propositions 3.1 and 3.2 in the paper \cite{Aubert} are essential in this process. 

Let $(B,[\;,\;]_0,\prs_0)$ be a pseudo-Riemannian flat Lie algebra,
$\xi,D:B\too B$ two
endomorphisms of $B$, $b_0\in B$ and $\mu\in\R$ such that:
\begin{enumerate}
	\item \label{enu1}$\xi$ is a 1-cocycle of  $(B,[\;,\;]_0)$ with respect to the
	representation $\mathrm{L}:B\too\mathrm{End}(B)$ defined by the left multiplication
	associated to the Levi-Civita product, i.e., for any $a,b\in B$,
	\begin{equation}\label{eq3}
	\xi([a,b])=\mathrm{L}_a\xi(b)-\mathrm{L}_b\xi(a),
	\end{equation}
	
	\item $D-\xi$ is skew-symmetric with respect to $\prs_0$,
	\begin{equation} \label{eq5}
	[D,\xi]=\xi^2-\mu\xi-\mathrm{R}_{b_0},
	\end{equation}and for any $a,b\in B$
	\begin{equation} \label{eq6}
	a.\xi(b)-\xi(a.b)=D(a).b+a.D(b)-D(a.b).
	\end{equation}
	
\end{enumerate}
We call  $(\xi,D,\mu,b_0)$ satisfying the two conditions above
\emph{admissible}.\\ Given $(\xi,D,\mu,b_0)$ admissible,
we endow the vector space $\G=\R e\oplus B\oplus\R \bar{e}$  with the inner product
$\prs$ which extends $\prs_0$, for which $\mathrm{span}\{e,\bar{e}\}$ and $B$ are
orthogonal,
$\langle e,e\rangle=\langle \bar{e},\bar{e}\rangle=0$ and $\langle e,\bar{e}\rangle=1$.
We define
also on
$\G$ the bracket
\begin{equation}\label{bracket}[\bar{e},e]=\mu e,\; [\bar{e},a]=D(a)-\langle
b_0,a\rangle_0e\esp[a,b]=[a,b]_0+\langle(\xi-\xi^*)(a),b\rangle_0e,\end{equation}where
$a,b\in B$
and
$\xi^*$ is the adjoint of $\xi$ with respect to $\prs_0$. Then $(\G,[\;,\;],\prs)$ is a
flat pseudo-Euclidean  Lie algebra  called \emph{double extension}
of $(B,[\;,\;]_0,\prs_0)$ according to $(\xi,D,\mu,b_0)$. Using this method, Aubert and Medina characterize a flat Lorentzian
nilpotent Lie algebras. They show that $(\G,\prs)$ is a flat Lorentzian nilpotent Lie
algebra if and only if $(\G,\prs)$ is a double extension
of an Euclidean abelian Lie algebra according to $\mu=0$, $D=\xi$ and
$b_0$ where $D^2=0$.

\begin{theo}\label{twosided}
	Let $(\G,\prs)$ be a flat pseudo-Euclidean nilpotent Lie algebra of signature $(2,n-2)$. Then, for any $e\in Z(\G)\cap Z(\G)^\bot$, $\mathrm{L}_e=\mathrm{R}_e=0$. Moreover, $Z(\G)+Z(\G)^\bot$ is a two-sided ideal with respect to the Levi-Civita product.
\end{theo}

\begin{proof} Recall that $[\G,\G]^\perp=\{u\in\G,\mathrm{R}_u=\mathrm{R}_u^*\}$,
	put $\mathfrak{a}=Z(\G)+Z(\G)^\bot$ and consider $N(\G)=\{v\in\G/\ \mathrm{L}_uv=0,\ \forall u\in Z(\G)\}$ and $\mathfrak{h}_0$ its orthogonal. We have seen in Lemma \ref{crucial} that both $N(\G)$ and $\mathfrak{h}_0$ are left ideals and $\mathfrak{h}_0$ is totally isotropic. We have seen that if $\dim\mathfrak{h}_0\leq1$ then $N(\G)=\G$ and hence any vector $e\in Z(\G)\cap Z(\G)^\bot$ satisfies the conditions required. Suppose that $\dim \mathfrak{h}_0=2$. We claim that $Z(\G)\cap Z(\G)^\bot\subset\mathfrak{h}_0$. This is a consequence of the fact that $Z(\G)\cap Z(\G)^\bot\subset Z(\G)\subset N(\G)$ and the fact that $N(\G)/\mathfrak{h}_0$ is Euclidean. We distinguish two cases:
	\begin{enumerate}
		\item $Z(\G)\cap Z(\G)^\bot=\mathfrak{h}_0$ and hence $\mathfrak{a}=N(\G)$. We have that $\G.N(\G)\subset N(\G)$ and for any $u\in N(\G)$, $w\in\G$ and $v\in\mathfrak{h}_0$, $v.u=u.v=0$ and hence $\langle u.w,v\rangle=0$. This implies that $N(\G)$ is an ideal for the Lie bracket and, according to Lemma \ref{codimension}, $[\G,\G]\subset N(\G)$. We deduce that $Z(\G)\cap Z(\G)^\bot\subset[\G,\G]^\bot$ and hence for any $e\in Z(\G)\cap Z(\G)^\bot$, $\mathrm{L}_e$ is both skew-symmetric and symmetric and hence $\mathrm{L}_e=\mathrm{R}_e=0$.
		\item $\dim Z(\G)\cap Z(\G)^\bot=1$. Since $Z(\G)\cap Z(\G)^\bot\subset\mathfrak{h}_0$, we have $N(\G)\subset\mathfrak{a}$ and $\mathfrak{a}=\G_0+\R y$. We have $\G.\G_0\subset\mathfrak{a}$ and for any $u\in\G_0$, $w\in\G$ and $v\in Z(\G)\cap Z(\G)^\bot$, $v.u=u.v=0$ and hence $\langle u.w,v\rangle=0$. Thus $N(\G).\G\subset\mathfrak{a}$. Moreover, for any $v\in Z(\G)\cap Z(\G)^\bot$, $\langle y.y,v\rangle=0$ and then $y.y\in\mathfrak{a}$. In particular, $\mathfrak{a}.\mathfrak{a}\subset\mathfrak{a}$ and hence $\mathfrak{a}$ is a subalgebra. According to Lemma \ref{codimension}, $[\G,\G]\subset\mathfrak{a}$ and hence $Z(\G)\cap Z(\G)^\bot\subset [\G,\G]^\bot$. This implies that for any $e\in Z(\G)\cap Z(\G)^\bot$, $\mathrm{L}_e=\mathrm{R}_e=0$ and $\mathfrak{a}$ is a two-sided ideal.\qedhere
	\end{enumerate}
\end{proof}
\begin{theo}\label{double}
	Let $(\G,\prs)$ be a flat pseudo-Euclidean nilpotent Lie algebra of signature $(2,n-2)$. Then $(\G,\prs)$ is a double extension of a flat Lorentzian nilpotent Lie algebra, according to $\mu=0$, $D$, $\xi$ and $b_0$ where $D$ is a nilpotent endomorphism.
\end{theo}
\begin{proof}
	Let $e$ be a non-null vector in $Z(\G)\cap Z(\G)^\bot$ and put $\mathrm{I}=\R e$. According to Theorem \ref{twosided}, $\mathrm{I}$ is a totally isotropic two-sided ideal with respect to the Levi-Civita product. Moreover, $\mathrm{I}^\bot$ is also a two sided ideal. Then, according to \cite{Aubert}, $(\G,\prs)$ is a double extension of flat Lorentzian Lie algebra $(B,\prs_B)$.  From \eqref{bracket} and the fact that $\G$ is nilpotent we deduce that $D$ is a nilpotent endomorphism, and $B$ is a nilpotent Lie algebra.
\end{proof}
\begin{rem}
	According to \cite{Aubert}, flat Lorentzian nilpotent Lie algebra are double extension of abelian Euclidean Lie algebras. Then flat pseudo-Euclidean nilpotent Lie algebras of signature $(2,n-2)$ are obtained by applying twice the double extension process, starting from abelian Euclidean Lie algebras.
\end{rem}
\begin{exem}
Let $(\G,\prs)$ be a 4-dimensional flat  pseudo-Euclidean nilpotent Lie algebras of signature $(2,2)$. According to theorem \ref{double}, $(\G,\prs)$ is a double extension of a 2-dimensional abelian Lorentzian Lie algebra $(B,\prs_B)$ with $D^2=0$. The conditions \eqref{eq3}-\eqref{eq6} are equivalent to $[D,\xi]=\xi^2$ and $D-\xi$ is skew-symmetric, which implies that $D=\xi$. Then there exists a basis $\{e_1,e_2\}$ of $B$ such that the matrix of $D$ in this basis has the form \[
\left(\begin{array}{cc}0&\alpha\\0&0\end{array}\right),\ \mbox{where }\alpha\in\R.
\]  Let $\prs_B$ be any Lorentzian metric in $B$. Then according to \eqref{bracket}, $\G=\mbox{span}\{\bar{e},e,e_1,e_2\}$ with the non vanishing Lie brackets
\[
[\bar{e},e_1]=\beta e,\ [\bar{e},e_2]=\alpha e_1+\gamma e,\ [e_1,e_2]=\delta e, \mbox{where }\alpha,\ \beta,\ \gamma,\ \delta\in\R,
\]
and the metric in $\G$ is an extension orthogonal of $\prs_B$ such that $\langle \bar{e},\bar{e}\rangle=\langle e,e\rangle=0$ and $\langle\bar{e},e\rangle=1$. It is easy to show that $\G$ is isomorphic to one of the following Lie algebras:
\begin{itemize}
	\item $\R^4$: The 4-dimensional abelian Lie algebra (if $\alpha=\beta=\gamma=\delta=0$).
	\item $\mathcal{H}_3\oplus\R$: The extension trivial of $\mathcal{H}_3$ (if $\alpha=0$ and $(\beta,\gamma)\neq(0,0)$ or $\alpha\neq0$ and $\beta=\delta=0$).
	\item The 4-dimensional filiform Lie algebra: $[\bar{e},e_1]=e,\ [\bar{e},e_2]=e_1$ (If $\alpha\neq0$ and $(\beta,\delta)\neq(0,0)$).
	
\end{itemize}

\end{exem}

\section{Flat pseudo-Euclidean 2-step nilpotent Lie algebras}
A 2-step nilpotent Lie algebra is a non-abelian Lie algebra $\G$ which satisfies $[\G,\G]\subset Z(\G)$. Let $(\G,\prs)$ be a flat pseudo-Euclidean 2-step nilpotent Lie algebra. In \cite{guediri}, the author showed that if the metric $\prs$ is Lorentzian, then $\G$ is an extension trivial of $\mathcal{H}_3$, where $\mathcal{H}_3$ is a 3-dimensional Heisenberg Lie algebra. Let us studies some properties of $(\G,\prs)$ in other signatures.

We consider $N(\G)=\bigcap_{u\in Z(\G)}\ker\mathrm{L}_u$, and $\h_0:=N(\G)^\perp$. According to Lemma \ref{crucial}, $\h_0\subset N(\G)$. If $N(\G)\neq\G$, then $N(\G)$ is degenerate. On the other hand, for any $z\in Z(\G)$, $a\in N(\G)$ and $u\in\G$ we have
	\[ \langle u.a,z\rangle=-\langle a,z.u\rangle=\langle z.a,u\rangle=-\langle a.u,z\rangle=0. \]
	This implies that $\G.N(\G)\subset Z(\G)^\perp$ and $N(\G).\G\subset Z(\G)^\perp$. Thus 
	\begin{equation}\label{gg0}
	[\G,N(\G)]\subset Z(\G)\cap Z(\G)^\perp.
	\end{equation}

\begin{pr}\label{dege}
	Let $(\G,\prs)$ be a flat pseudo-Euclidean 2-step nilpotent Lie algebra. Then
	\begin{enumerate}
	\item $Z(\G)$ is degenerate.
	\item For any $e\in Z(\G)\cap Z(\G)^\bot$, $\mathrm{L}_{e}=\mathrm{R}_{e}=0.$
	\item For any $x,y\in Z(\G)^\bot$, $\langle[x,y],[x,y]\rangle=0.$
	\end{enumerate}
\end{pr} 
\begin{proof}
\begin{enumerate}
\item Suppose that $Z(\G)$ is non degenerate, i.e., $Z(\G)\cap Z(\G)^\perp=\{0\}$. 
\begin{itemize}
\item If $\G=N(\G)$ then according to \eqref{gg0}, $[\G,\G]=0$ which is impossible. 
\item If $\G\not=N(\G)$ then $[\G,N(\G)]=0$ and hence $N(\G)=Z(\G)$ which is impossible since $N(\G)$ is degenerate.
\end{itemize}
\item Let $e\in Z(\G)\cap Z(\G)^\bot$. Since  $Z(\G)^\perp\subset[\G,\G]^\perp$, then according to \eqref{eq7}, $\mathrm{L}_{e}=\mathrm{R}_{e}$ is both symmetric and skew-symmetric and hence must vanish.
\item According to \eqref{lc}, we have for any $x,y\in Z(\G)^\bot$ $x.y=\frac{1}{2}[x,y]$. Using \eqref{eqs1}, we have $[x,y].x=x.(y.x)-y.(x.x)$, then $[x,y].x=0$. In particular $\langle[x,y].x,y\rangle=0$. Since $\mathrm{L}_x$ is skew-symmetric, thus $\langle[x,y],[x,y]\rangle=0.$ 
\end{enumerate}
\end{proof}
\begin{pr}
Let $(\G,\prs)$ be a flat pseudo-Euclidean 2-step nilpotent Lie algebra. Then $(\G,\prs)$ is obtained by a sequence of double extension, starting from an abelian pseudo-Euclidean Lie algebra.
\end{pr}
\begin{proof}
Let e be a non-null vector in $Z(\G)\cap Z(\G)^\bot$. Since $\mathrm{L}_{e}=\mathrm{R}_{e}=0$, then $\mathrm{I}=\R e$ is a totally isotropic two sided ideal, and $\mathrm{I}^\bot$ is also a two sided ideal. Thus, $(\G,\prs)$ is a double extension of a pseudo-Euclidean Lie algebra $(B_1,\prs_1)$. According to \eqref{bracket}, $B_1$ is either abelian or 2-step nilpotent. If $B_1$ is 2-step nilpotent, then it's also a double extension of $(B_2,\prs_2)$. Since a 2-step nilpotent Lie algebra can not admit a flat Euclidean metric, then there exists $k\in\N^*$ such that $B_k$ is abelian.  
\end{proof}
\begin{pr}\label{abelian}
	Let $(\G,\prs)$ be a flat pseudo-Euclidean 2-step nilpotent Lie algebra of signature $(p,p+q)$. If $\dim(Z(\G)\cap Z(\G)^\perp)=p$ then $Z(\G)^\perp$ is abelian.
\end{pr}
\begin{proof}
Let $\{e_1,\ldots,e_p\}$ be a basis of $Z(\G)\cap Z(\G)^\perp$, then we can whrite $Z(\G)=Z_1\oplus\mbox{span}\{e_1,\ldots,e_p\}$ where $(Z_1,\prs_{/_{Z_1\times Z_1}})$ is euclidean. In $Z_1^\bot$ we can choose a totaly isotropic subspace $\mbox{span}\{\bar{e}_1,\ldots,\bar{e}_p\}$ such that, $\langle e_i,\bar{e}_j\rangle=0$ for $i\neq j$, and $\langle e_i,\bar{e}_i\rangle=1$. Let $B_1$ be the orthogonal of $Z_1\oplus\mbox{span}\{e_1,\ldots,e_p\}\oplus\mbox{span}\{\bar{e}_1,\ldots,\bar{e}_p\}$. Thus we get a decomposition 
\begin{equation}\label{decomposit}
\G=Z_1\oplus\mbox{span}\{e_1,\ldots,e_p\}\oplus B_1\oplus\mbox{span}\{\bar{e}_1,\ldots,\bar{e}_p\}.
\end{equation} 
We have $Z(\G)^\bot=B_1\oplus\mbox{span}\{\bar{e}_1,\ldots,\bar{e}_p\}$, and $(B_1,\prs_{/_{B_1\times B_1}})$ is euclidean. Let $x,y\in Z(\G)^\bot$, $z\in Z(\G)$ and $k\in\{1,\ldots,p\}$. We have $\langle [x,y],[x,y]\rangle=0$, $z.\bar{e}_k\in Z(\G)^\bot$ and $\langle z.\bar{e}_k,z.\bar{e}_k\rangle=0$. Since $Z_1$ and $B_1$ are euclidean, then $[x,y]$ and $z.\bar{e}_k$ are in $Z(\G)\cap Z(\G)^\perp$. Thus
\[
0=\langle z.\bar{e}_k,x\rangle=-\frac{1}{2}\langle[\bar{e}_k,x],z\rangle,
\]
which implies that $[\bar{e}_k,x]\in Z(\G)\cap Z(\G)^\perp$. Using the flatness of the metric, then $[\bar{e}_k,x].x=\bar{e}_k.(x.x)-x.(\bar{e}_k.x)$, thus $x.(\bar{e}_k.x)=0$. Let $\{b_1,\ldots,b_r\}$ be an orthonormal basis of $B_1$. We have $\bar{e}_k.x\in Z(\G)^\bot$, and $\langle \bar{e}_k.x,b_i\rangle=-\frac{1}{2}\langle [x,b_i],\bar{e}_k\rangle$. Then $\bar{e}_k.x=e_0-\frac{1}{2}\sum_{i=1}^r\langle[x,b_i],\bar{e}_k\rangle b_i$, where $e_0\in Z(\G)\cap Z(\G)^\perp$. Using the fact that $x.b_i=\frac{1}{2}[x,b_i]$, then
\begin{eqnarray*}
x.(\bar{e}_k.x)&=&-\frac{1}{2}\sum_{i=1}^r\langle[x,b_i],\bar{e}_k\rangle x.b_i\\
&=&-\frac{1}{4}\sum_{i=1}^r\langle[x,b_i],\bar{e}_k\rangle [x,b_i].
\end{eqnarray*}	
Finally, we get $\sum_{i=1}^r\langle[x,b_i],\bar{e}_k\rangle^2=0$. Since $[x,b_i]\in Z(\G)\cap Z(\G)^\perp$, thus $Z(\G)^\bot$ is abelian.
\end{proof}
Suppose that $\dim Z(\G)\cap Z(\G)^\perp=1$. Then the decomposition \eqref{decomposit} becomes
\begin{equation}\label{decom2}
\G=Z_1\oplus\R e\oplus B_1\oplus\R\bar{e},
\end{equation}
and the restriction of $\prs$ to $Z_1$ and $B_1$ is nondegenerate.
\begin{pr}\label{definite}
Let $(\G,\prs)$ be a flat pseudo-Euclidean 2-step nilpotent Lie algebra such that $\dim(Z(\G)\cap Z(\G)^\perp)=1$. With notations as in \eqref{decom2}, if the restriction of the metric $\prs$ to $B_1$ is positive or negative definite, then $\dim B_1=1$, and $\G$ is an extension trivial of $\mathcal{H}_3$, where $\mathcal{H}_3$ is the 3-dimensional Heisenberg Lie algebra.     
\end{pr}
\begin{proof}
Let $z\in Z(\G)$, and $b\in B_1$. We have $z.\bar{e}\in B_1$ and $z.b\in Z(\G)^\bot$. Since $\prs_{/_{B_1\times B_1}}$ is positive definite or negative definite and $\langle z.\bar{e},z.\bar{e}\rangle=0$, then $z.\bar{e}=0$. Thus $\langle z.b,\bar{e}\rangle=0$, which implies that $z.b\in B_1$. Using the same argument, then we can conlude that $z.b=0$, and $\mathrm{L}_z=0$ for any $z\in Z(\G)$.  Let $x,y\in B_1$. We have for any $z\in Z(\G)$
\[
\langle [x,y],z\rangle=2\langle x.y,z\rangle=0,
\]
thus $[x,y]=\alpha e$, where $\alpha\in\R$. Using the flatness of the metric, then we get $[\bar{e},x].x=\bar{e}.(x.x)-x.(\bar{e}.x)$, thus $x.(\bar{e}.x)=0$. Let $\{b_1,\dots,b_r\}$ be an orthonormal basis of $B_1$. Then 
\[
\bar{e}.x=\beta e\mp\frac{1}{2}\sum_{i=1}^r\langle[x,b_i],\bar{e}\rangle b_i
\]
 where $\beta\in\R$. Thus 
\[ 
x.(\bar{e}.x)=\mp\frac{1}{4}\sum_{i=1}^r\langle[x,b_i],\bar{e}\rangle [x,b_i]=0,
\]
which implies that $B_1$ is abelian. On the other hand, we have for any $z\in Z(\G)$, 
\[
0=\langle z.\bar{e},x\rangle=-\frac{1}{2}\langle[\bar{e},x],z\rangle,
\]
 thus $[\bar{e},x]\in Z(\G)\cap Z(\G)^\perp$.
Put $[\bar{e},b_i]=\alpha_i e$, where $\alpha_i\in\R^*$ for any $i\in\{1,\ldots,r\}$. In fact, if $\alpha_i=0$ then $b_i\in Z(\G)$, which contradicts the fact that $Z(\G)\cap B_1=\{0\}$. Suppose that $\dim B_1>1$. For any $i\in\{2,\ldots,r\}$, we put $b'_i=b_i-\frac{\alpha_i}{\alpha_1}b_1$, thus $[\bar{e},b'_i]=0$ and $b'_i\in Z(\G)$ which is a Contradiction. Then $\dim B_1=1$ and the only non vanishing brackets in $\G$ is $[\bar{e},b_1]=\alpha_1 e$, thus $\G$ is an extension trivial of $\mathcal{H}_3$. 
\end{proof}

\section{Flat pseudo-Euclidean 2-step nilpotent Lie algebras of signature $(2,n-2)$}
Let us start by an example which play an important role in this section. Let $\mathrm{L}_6^4$ be a 6-dimenional Lie algebra defined by the non vanishing Lie brackets, giving in the basis $\{x_1,\ldots,x_6\}$ by
\[
[x_1,x_2]=x_5\ ,\ [x_1,x_3]=[x_2,x_4]=x_6.
\]
This Lie algebra appear in the classification of 2-step nilpotent Lie algebras of dimension 6, as for example in \cite[pp.3]{Abiev}, or in \cite[pp.97]{revoy}, where it is denoted by $L_{6,3}$.

It is clear that this Lie algebra admits no flat Euclidean or Lorentzian metrics. However, $\mathrm{L}_6^4$ admits a flat pseudo-Euclidean metrics of signature $(2,n-2)$. In fact, let $\prs_0$ be a pseudo-Euclidean metric of signature $(2,4)$ defined in the basis $\{x_1,\ldots,x_6\}$ by the matrix
\[
\prs_0=\left(\begin{array}{cccccc}0&0&0&0&a&1\\0&0&b&c&0&0\\0&b&0&0&0&0\\0&c&0&d&0&0\\a&0&0&0&\frac{1}{3d}&0\\1&0&0&0&0&0\end{array}\right)\]
where $a,c\in\R$, $b\in\R^*$ and $d>0$. A straightforward calculations using \eqref{lc} shows that, the only non vanishing Levi-Civita products are
\begin{eqnarray*}
&&x_1.x_1=-\frac{1}{b}x_2-\left(\frac{a}{b}+\frac{c^2}{b^2d}\right)x_3+\frac{c}{bd}x_4,\  x_1.x_2=\frac{c}{2bd}x_3-\frac{1}{2d}x_4+\frac{1}{2}x_5+\frac{a}{2}x_6,\\
&&x_1.x_3=x_6,\ x_1.x_4=\frac{1}{2b}x_3,\ x_1.x_5=x_5.x_1=-\frac{1}{6bd}x_3,\ x_2.x_4=\frac{1}{2}x_6,\ x_2.x_5=x_5.x_2=\frac{1}{6d}x_6,\\
&&x_2.x_1=\frac{c}{2bd}x_3-\frac{1}{2d}x_4-\frac{1}{2}x_5+\frac{a}{2}x_6,\ x_4.x_1=\frac{1}{2b}x_3,\ x_4.x_2=-\frac{1}{2}x_6.\\
\end{eqnarray*}
One can verify that for any $x,y\in \mathrm{L}_6^4$, we have $\mathrm{L}_{[x,y]}=\left[\mathrm{L}_x,\mathrm{L}_y\right]$, which shows that $(\mathrm{L}_6^4,\prs_0)$ is flat. The following Theorem shows that this example, is the only non trivial one such that $\dim Z(\G)\cap Z(\G)^\bot=1$.

Let $(\G,\prs)$ be a flat pseudo-Euclidean 2-step nilpotent Lie algebra of signature $(2,n-2)$. According to theorem \ref{degenerate}, the dimension of $Z(\G)\cap Z(\G)^\bot$ is 1 or 2.
\begin{theo}\label{dim1}
	A 2-step nilpotent Lie algebra $\G$ admits a flat pseudo-Euclidean metric of signature $(2,n-2)$ such that $\dim Z(\G)\cap Z(\G)^\bot=1$ if and only if $\G$ is an extension trivial of $\mathcal{H}_3$ or $\G$ is an extension trivial of $\mathrm{L}_6^4$. Furthermore, in the second case, the restriction of the metric to $\mathrm{L}_6^4$ is giving by $\prs_0$.
\end{theo}
\begin{proof}
If $\dim Z(\G)\cap Z(\G)^\bot=1$, then we can split $\G$ as
\[
\G=Z_1\oplus\R e\oplus B_1\oplus\R\bar{e},
\]
where $Z(\G)=Z_1\oplus\R e$, $Z(\G)^\bot=\R e\oplus B_1$, $\mbox{span}\{e,\bar{e}\}=\left(Z_1\oplus B_1\right)^\bot$, $\langle e,e\rangle=\langle \bar{e},\bar{e}\rangle=0$ and $\langle e,\bar{e}\rangle=1$. We have two cases:\\
{\bf First case:} $\prs/_{B_1\times B_1}$ is positive or negative definite. Then according to proposition \ref{definite}, $\dim B_1=1$ and $\G$ is an extension trivial of $\mathcal{H}_3$.\\
{\bf Second case:} $\prs/_{B_1\times B_1}$ is Lorentzian. Then $\dim B_1\geq2$. For any $z,z'\in Z(\G)$, we have $\langle z.\bar{e},z'.\bar{e}\rangle=0$, then $\mathrm{R}_{\bar{e}}(Z(\G))$ is a totally isotropic subspace. Since $\mathrm{R}_{\bar{e}}(Z(\G))\subset B_1$ and $(B_1,\prs/_{B_1\times B_1})$ is Lorentzian, then there exists an isotropic vector $b_0\in B_1$ and a covector $\lambda\in Z(\G)^*$ such that $z.\bar{e}=\lambda(z)b_0$ for any $z\in Z(\G)$.\\ 
Let $x,y\in Z(\G)^\bot$. Recall that $x.y=\frac{1}{2}[x,y]$ and $\langle [x,y],[x,y]\rangle=0$. Since $Z_1$ is Euclidean then $[x,y]\in Z(\G)\cap Z(\G)^\bot$. Choose a basis $\{b_0,\bar{b},b_1,\ldots,b_r\}$ of $B_1$ such that $\{b_1,\ldots,b_r\}$ is orthonormal, $\mbox{span}\{b_0,\bar{b}\}$ and $\mbox{span}\{b_1,\ldots,b_r\}$ are orthogonal, $\bar{b}$ is isotropic and $\langle b_0,\bar{b}\rangle=1$. Then for any $i\in\{0,1,\ldots,r\}$, we have from \eqref{lc}
\[
\langle [\bar{e},b_i].\bar{e},b_i\rangle=-\frac{1}{2}\langle [\bar{e},b_i],[\bar{e},b_i]\rangle.
\] 
On the other hand, we have $\langle [\bar{e},b_i].\bar{e},bi\rangle=\langle\lambda\left([\bar{e},b_i]\right)b_0,b_i\rangle=0$, then $[\bar{e},b_i]\in Z(\G)\cap Z(\G)^\bot$.\\
We can write from the condition of flatness, for any $x,y,z\in\G$
\begin{equation}\label{flat}
[x,y].z=x.(y.z)-y.(x.z).
\end{equation} 
If we take $x=\bar{e}$ and $y=z=b_0$, we get $b_0.(\bar{e}.b_0)=0$. Let $i\in\{0,1,\ldots,r\}$, since $\bar{e}.b_0\in Z(\G)^\bot$ and $\langle \bar{e}.b_0,b_i\rangle=-\frac{1}{2}\langle[b_0,b_i],\bar{e}\rangle$, thus 
\[
\bar{e}.b_0=\alpha e+\beta b_0-\frac{1}{2}\sum_{i=1}^r\langle[b_0,b_i],\bar{e}\rangle,
\]
 where $\alpha,\beta\in\R$. It follows that $b_0.(\bar{e}.b_0)=-\frac{1}{4}\sum_{i=1}^r\langle[b_0,b_i],\bar{e}\rangle[b_0,b_i]$, thus $\sum_{i=1}^r\langle[b_0,b_i],\bar{e}\rangle^2=0$ which implies that $[b_0,b_i]=0$ for any $i\in\{0,1,\ldots,r\}$.\\
If we take in \eqref{flat}, $x=\bar{e}$, $y=b_0$ and $z=\bar{b}$ we get $b_0.(\bar{e}.\bar{b})=0$. Using the fact that $b_0.u=0$ for any $u\in Z(\G)$, we deduce that $b_0.(\bar{e}.\bar{b})=-\frac{1}{4}\langle[\bar{b},b_0],\bar{e}\rangle[\bar{b},b_0]$, thus $[\bar{b},b_0]=0$. Similarly, for any $i\in\{1,\ldots,r\}$, if we take in \eqref{flat}, $x=\bar{e}$ and $y=z=e_i$ we get
\[
0=b_i.(\bar{e}.b_i)=-\frac{1}{4}\sum_{j=1}^r\langle[b_i,b_j],\bar{e}\rangle[b_i,b_j],
\]
thus $[b_i,b_j]=0$ for any $i,j\in\{1,\ldots,r\}$. It follows that $\mbox{span}\{b_0,b_1,\ldots,b_r\}$ is abelian and $[b_0,\bar{b}]=0$. We put
\[
[\bar{e},b_i]=\alpha_ie,\ [\bar{e},\bar{b}]=\alpha e+z_0,\ [\bar{b},b_i]=\beta_ie,
\]
where $\alpha_i,\beta_i,\alpha\in\R$, $z_0\in Z_1$ and $i=0,1,\ldots,r$. If we take in \eqref{flat}, $x=\bar{e}$ and $y=z=\bar{b}$ we get $z_0.\bar{b}=-\bar{b}(\bar{e}.\bar{b})$, then $\frac{3}{2}z_0.\bar{b}-\frac{1}{2}\sum_{i=1}^r\beta_i^2e=0$, which implies that 
\begin{equation}\label{eqq}
3\langle z_0,z_0\rangle=\sum_{i=1}^r\beta_i^2.
\end{equation} 
We have $\dim B_1\geq3$. In fact, if $\dim B_1=2$ then $B_1=\mbox{span}\{b_0,\bar{b}\}$ and \eqref{eqq} implies that $z_0=0$. Then the Lie brackets are reduced to $[\bar{e},b_i]=\alpha_ie$ and $[\bar{e},\bar{b}]=\alpha e$, and as in the proof of proposition \ref{definite} we can deduce that $\dim B_1=1$, which is a contradiction. The same argument shows that $z_0\neq0$. Then there exists $i\in\{1,\ldots,r\}$ such that $\beta_i\neq0$. To simplify, we can suppose that $\beta_1\neq0$, and we have also $\alpha_0\neq0$ because $b_0\notin Z(\G)$.\\
Let us show that $\dim B_1=3$. In fact, if $\dim B_1\geq4$, then we put for any $i\geq4$, \[b_i'=b_i-\frac{\beta_i}{\beta_1}b_1-\left(\frac{\alpha_i\beta_1-\alpha_1\beta_i}{\alpha_0\beta_1}\right)b_0,\] and we can verify easly that $[b_i',x]=0$ for any $x\in\G$. Thus $b_i'\in Z(\G)$ which contradicts the fact that $Z(\G)\cap B_1=\{0\}$. We put $x_1=\bar{e}$, $x_2=\bar{b}$, $x_3=\frac{b_0}{\alpha_0}$, $x_4=\frac{1}{\beta_1}b_1-\frac{\alpha_1}{\beta_1\alpha_0}b_0$, $x_5=\alpha e+z_0$ and $x_6=e$. Then the only non vanishing brackets on $\G$ are
\[
[x_1,x_2]=x_5,\ [x_1,x_3]=[x_2,x_4]=x_6.
\]
It follows that $\G$ is an extension trivial of $\mathrm{L}_6^4$. Furthermore, with the condition \eqref{eqq}, one can verify that the restriction of the metric to $\mathrm{L}_6^4$ is given by $\prs_0$. Conversely, if $\G$ splits orthogonaly into $\G=Z_1\oplus\mathrm{L}_6^4$, where $Z_1\subset Z(\G)$ and the restriction of the metric to $\mathrm{L}_6^4$ is $\prs_0$, and the restriction to $Z_1$ is Euclidean, then $(\G,\prs)$ is a flat pseudo-Euclidean 2-step nilpotent Lie algebra of signature $(2,n-2)$ and $\dim Z(\G)\cap Z(\G)^\bot=1$. 
\end{proof}
\begin{co}
The Heisenberg Lie algebra $\mathcal{H}_{2k+1}$ admits a flat pseudo-Euclidean metric of signature $(2,n-2)$ if and only if $k=1$.
\end{co}
\begin{proof}
Since $Z(\mathcal{H}_{2k+1})=1$, then if $\G=\mathcal{H}_{2k+1}$ admits such metric then we have $\dim Z(\G)\cap Z(\G)^\bot=1$. This gives the result, by virtue of Theorem \ref{dim1}.
\end{proof}
\begin{rem}
In theorem \ref{dim1}, if $\G$ is an extension trivial of $\mathcal{H}_3$, then $\G=Z_1\oplus\mathcal{H}_3$ and the metric $\prs$ has one of the following form:
\begin{itemize}
\item The restriction of $\prs$ to $Z_1$ is Euclidean and its restriction to $\mathcal{H}_3$ is given by the matix
\[
\left(\begin{array}{ccc}0&0&\alpha\\0&-1&0\\\alpha&0&0\end{array}\right),\ \mbox{where }\alpha\in\R.
\]
\item The restriction of $\prs$ to $Z_1$ is Lorentzian and its restriction to $\mathcal{H}_3$ is given by the matix
\[
\left(\begin{array}{ccc}0&0&\alpha\\0&1&0\\\alpha&0&0\end{array}\right),\ \mbox{where }\alpha\in\R.
\]
\end{itemize}
\end{rem}
\begin{theo}\label{dim2}
	A 2-step nilpotent Lie algebra $\G$ admits a flat pseudo-Euclidean metric $\prs$ of signature $(2,n-2)$ such that $\dim Z(\G)\cap Z(\G)^\bot=2$ if and only if there exist an orthonormal vectors $\{b_1,\ldots,b_k\}$ in $\G$, a linearly independent isotropic vectors $\{e_1,\bar{e}_1,e_2,\bar{e}_2\}$ in $\{b_1,\ldots,b_k\}^\bot$, where $\langle e_1,e_2\rangle=\langle e_1,\bar{e}_2\rangle=\langle \bar{e}_1,e_2\rangle=\langle \bar{e}_1,\bar{e}_2\rangle=0$ and $\langle e_1,\bar{e}_1\rangle=\langle e_2,\bar{e}_2\rangle=1$, such that for any $i\in\{1,\dots,k\}$ the only non vanishing brackets are
	\begin{eqnarray}\nonumber
	\ [\bar{e}_1,\bar{e}_2]&=&z_0,\\\label{cond1}
	\ [\bar{e}_1,b_i]&=&\alpha_ie_1+\beta_ie_2,\\\nonumber
	\ [\bar{e}_2,b_i]&=&\gamma_ie_1+\delta_ie_2,
	\end{eqnarray}
	where $\alpha_i,\beta_i,\gamma_i,\delta_i\in\R$, and 
	\begin{equation}\label{cond}
	3\langle z_0,z_0\rangle=\sum_{i=1}^k(\gamma_i+\beta_i)^2-4\alpha_i\delta_i.
\end{equation}
\end{theo}
\begin{proof}
According to proposition \ref{abelian}, $Z(\G)^\bot$ is abelian, and we can split $\G$ into 
\begin{equation}\label{split}
\G=Z_1\oplus\mbox{span}\{e_1,e_2\}\oplus B_1\oplus\mbox{span}\{\bar{e}_1,\bar{e}_2\},
\end{equation}
where $Z(\G)=Z_1\oplus\mbox{span}\{e_1,e_2\}$, $Z(\G)^\bot=\mbox{span}\{e_1,e_2\}\oplus B_1$, $\left(Z_1\oplus B_1\right)^\bot=\mbox{span}\{e_1,e_2,\bar{e}_1,\bar{e}_2\}$, $\mbox{span}\{\bar{e}_1,\bar{e}_2\}$ is totally isotropic, $\langle e_1,e_2\rangle=\langle e_1,\bar{e}_2\rangle=\langle \bar{e}_1,e_2\rangle=\langle \bar{e}_1,\bar{e}_2\rangle=0$ and $\langle e_1,\bar{e}_1\rangle=\langle e_2,\bar{e}_2\rangle=1$.\\
In the proof of proposition \ref{abelian}, we have shown that for any $x,y\in Z(\G)^\bot$ and $k\in\{1,2\}$, $[x,y]$ and $[\bar{e}_k,x]$ are in $Z(\G)\cap Z(\G)^\bot$. Let $\{b_1,\ldots,b_r\}$ be an orthonormal basis of $B_1$. Then, the non vanishing brackets are: 
\begin{eqnarray*}
	\ [\bar{e}_1,\bar{e}_2]&=&z_0,\\
	\ [\bar{e}_1,b_i]&=&\alpha_ie_1+\beta_ie_2,\\
	\ [\bar{e}_2,b_i]&=&\gamma_ie_1+\delta_ie_2,
	\end{eqnarray*}
where $z_0\in Z(\G),\alpha_i,\beta_i,\gamma_i,\delta_i\in\R$ and $i=1,\ldots,r$. From \eqref{lc} and the Lie brackets above, we have for any $u\in Z(\G)$ and $v\in Z(\G)^\bot$, $u.v=0$. Recall that $(\G,\prs)$ is flat if and only if for any $x,y,z\in\G$
\begin{equation}\label{flatt}
\mathrm{L}_{[x,y]}(z)=\left[\mathrm{L}_x,\mathrm{L}_y\right](z).
\end{equation}
Let $x\in Z(\G)+Z(\G)^\bot$, $y,z\in\G$ and $i\in\{1,2\}$. We have $\langle y.z,e_i\rangle=0$, then $y.z\in Z(\G)+Z(\G)^\bot$. Thus $x.(y.z)=(y.z).x=0$. On the other hand, we have $x.y,\ y.x\in Z(\G)\cap Z(\G)^\bot$. Thus $(x.y).z=(y.x).z=0$. It follows that if one of the vectors $x,y$ or $z$ is in $Z(\G)+Z(\G)^\bot$, then \eqref{flatt} is satisfied. Thus $(\G,\prs)$ is flat if and only if 
\[
\mathrm{L}_{[\bar{e}_1,\bar{e}_2]}\bar{e}_1-\left[\mathrm{L}_{\bar{e}_1},\mathrm{L}_{\bar{e}_2}\right]\bar{e}_1=\mathrm{L}_{[\bar{e}_1,\bar{e}_2]}\bar{e}_2-\left[\mathrm{L}_{\bar{e}_1},\mathrm{L}_{\bar{e}_2}\right]\bar{e}_2=0.
\]
Straithforward calculations using \eqref{lc} give
\begin{eqnarray*}
&&z_0.\bar{e}_1=-\frac{1}{2}\langle z_0,z_0\rangle e_2,\ \bar{e}_2.\bar{e}_1=-\frac{1}{2}z_0-\frac{1}{2}\sum_{i=1}^r(\beta_i+\gamma_i)b_i,\ \bar{e}_1.\bar{e}_1=-\sum_{i=1}^r\alpha_ib_i,\\
&&\bar{e}_1b_i=\alpha_ie_1+\frac{1}{2}(\beta_i+\gamma_i)e_2,\ \bar{e}_2b_i=\frac{1}{2}(\beta_i+\gamma_i)e_1+\delta_ie_2.
\end{eqnarray*}  
Thus the condition $\mathrm{L}_{[\bar{e}_1,\bar{e}_2]}\bar{e}_1-\left[\mathrm{L}_{\bar{e}_1},\mathrm{L}_{\bar{e}_2}\right]\bar{e}_1=0$ is equivalent to \eqref{cond1}. Similarly, we show that the second condition $\mathrm{L}_{[\bar{e}_1,\bar{e}_2]}\bar{e}_2-\left[\mathrm{L}_{\bar{e}_1},\mathrm{L}_{\bar{e}_2}\right]\bar{e}_2=0$ is also equivalent to \eqref{cond1}. This completes the proof.
\end{proof}
\begin{co}
	If a 2-step nilpotent Lie algebra $\G$ admits a flat pseudo-Euclidean metric of signature $(2,n-2)$, then $\dim[\G,\G]\leqslant3$.
\end{co}
\section*{Examples}
In this section, we show that any 6-dimensional 2-step nilpotent Lie algebra, which is not an extension trivial of $\mathcal{H}_5$, admits a flat pseudo-Euclidean metric of signature $(2,n-2)$, where $\mathcal{H}_5$ is a 5-dimensional Heisenberg Lie algebra. For this, we use the table below which give all 6-dimensional 2-step nilpotent Lie algebras (see \cite[pp.3]{Abiev}). Note that $\mathcal{H}_5$ (resp. $\mathcal{H}_3$) is denoted in this table by $\mathrm{L}_5^4$ (resp. $\mathrm{L}_3$). 
\begin{center}
	\begin{tabular}{ll}
		\hline
		{\bf Lie algebra}&{\bf Nonzero commutators}\\
		\hline
		$L_3\oplus3L_1$&$[x_1,x_2]=x_3$\\
		$L_5^1\oplus L_1$&$[x_1,x_2]=x_3,\ [x_1,x_4]=x_5$\\
		$L_5^4\oplus L_1$&$[x_1,x_3]=x_5,\ [x_2,x_4]=x_5$\\
		$L_3\oplus L_3$&$[x_1,x_2]=x_3,\ [x_4,x_5]=x_6$\\
		$L_6^4$&$[x_1,x_2]=x_5,\ [x_1,x_3]=x_6,\ [x_2,x_4]=x_6$\\
		$L_6^5(-1)$&$[x_1,x_3]=x_5,\ [x_1,x_4]=x_6,\ [x_2,x_4]=x_5,\ [x_2,x_3]=-x_6$\\
		$L_6^3$&$[x_1,x_3]=x_6,\ [x_1,x_2]=x_4,\ [x_2,x_3]=x_5$\\
		\hline
	\end{tabular}
\end{center}
The result is evident for $\mathrm{L}_3\oplus3\mathrm{L}_1$ and $\mathrm{L}_6^4$ (theorem \ref{dim1}). Let $\prs$ be a pseudo-Euclidean metric of signature $(2,n-2)$ given by the matrix  
\[
\prs=\left(\begin{array}{cccccc}0&1&0&0&0&0\\1&0&0&0&0&0\\0&0&0&1&0&0\\0&0&1&0&0&0\\0&0&0&0&1&0\\0&0&0&0&0&1\end{array}\right)
\]
Using theorem \ref{dim2}, let us show that, for all those Lie algebras $\mathrm{L}_5^1\oplus\mathrm{L}_1$, $\mathrm{L}_3\oplus\mathrm{L}_3$, $\mathrm{L}_6^5(-1)$ and $\mathrm{L}_6^3$ there exists a basis $\mathbb{B}$ such that the metric given in $\mathbb{B}$ by $\prs$ is flat.
\begin{itemize}
\item For $\mathrm{L}_5^1\oplus\mathrm{L}_1$, with our notations we put $\mathbb{B}=\{e_1,\bar{e}_1,e_2,\bar{e}_2,z_0,b_1\}$ where $e_1=x_6$, $\bar{e}_1=x_1$, $e_2=x_5$, $\bar{e}_2=x_2$, $z_0=x_3$ and $b_1=x_4$. One can verify easly that in this basis, the Lie brackets and the metric verify the conditions \eqref{cond1} and \eqref{cond}, thus $\left(\mathrm{L}_5^1\oplus\mathrm{L}_1,\prs\right)$ is flat. 
\item For $\mathrm{L}_3\oplus\mathrm{L}_3$, we put $\mathbb{B}=\{e_1,\bar{e}_1,e_2,\bar{e}_2,b_1,b_2\}$ where $e_1=x_3$, $\bar{e}_1=x_1$, $e_2=x_6$, $\bar{e}_2=x_4$, $b_1=x_1$ and $b_2=x_5$.
\item For $\mathrm{L}_6^5(-1)$, we put $\mathbb{B}=\{e_1,\bar{e}_1,e_2,\bar{e}_2,b_1,b_2\}$ where $e_1=x_5+x_6$, $\bar{e}_1=x_1$, $e_2=-2x_6$, $\bar{e}_2=x_2$, $b_1=x_4$ and $b_2=-(3+\sqrt{15})x_4-x_3$.
\item For $\mathrm{L}_6^3$, we put $\mathbb{B}=\{e_1,\bar{e}_1,e_2,\bar{e}_2,z_0,b_1\}$ where $e_1=x_4$, $\bar{e}_1=x_1$, $e_2=\frac{4}{3}x_5$, $\bar{e}_2=x_3$, $z_0=x_6$ and $b_1=x_2$.
\end{itemize}
For $\G=\mathrm{L}_5^4\oplus\mathrm{L}_1$, it is clear that this algebra can not admit flat pseudo-Euclidean metric of signature $(2,n-2)$ such that $\dim Z(\G)\cap Z(\G)^\bot=1$ (theorem \ref{dim1}). Suppose that it admits such metric with $\dim Z(\G)\cap Z(\G)^\bot=2$ (theorem \ref{dim2}). We have $\dim[\G,\G]=1$ and $\dim Z(\G)=2$. Then $\dim Z(\G)^\bot=4$ and $\dim B_1=2$. Put $[\G,\G]=\R e_1$, thus the Lie brackets satisfy 
\[
[\bar{e}_1,\bar{e}_2]=\alpha e_1,\ [\bar{e}_1,b_i]=\alpha_ie_1,\ [\bar{e}_2,b_i]=\gamma_ie_1,\ i=1,2.
\]
The condition \eqref{cond} implies that $\gamma_1=\gamma_2=0$. Then $\alpha,\alpha_1,\alpha_2\in\R^*$. The fact that $\alpha=0$, for example, implies that $\bar{e}_2\in Z(\G)$. Put $b_2'=b_2-\frac{\alpha_2}{\alpha_1}b_1$, then $b_2'\in Z(\G)$, which is a contradiction. It follows that $\mathrm{L}_5^4\oplus\mathrm{L}_1$ can not admit flat pseudo-Euclidean metrics of signature $(2,n-2)$.

\bibliographystyle{elsarticle-num}

\end{document}